\documentclass[a4paper]{article}

\usepackage{amsmath}

\usepackage{amsthm}
\usepackage[dvips]{graphicx}

\usepackage{amscd}
\usepackage{epsfig}
 
\usepackage{color}

\usepackage{psfrag}
\usepackage{srcltx}

\usepackage{amssymb}
\usepackage{latexsym}
\newtheorem{Theorem}{Theorem}

\newtheorem{Proposition}[Theorem]{Proposition}
\newtheorem{Lemma}[Theorem]{Lemma}
\newtheorem{Corollary}[Theorem]{Corollary}

\theoremstyle{definition}
\newtheorem{definition}{Definition}
\newtheorem{remark}{Remark}
\newtheorem{notation}{Notation}

\begin{document}

\title{Invariant character varieties of hyperbolic knots with symmetries}

\author{Luisa Paoluzzi\footnote{Partially supported by ANR project 
12-BS01-0003-01}\quad
and Joan Porti\footnote{Partially supported by the 
Spanish Mineco through grant MTM2015-66165-P }} 
\date{\today}

\maketitle

\begin{abstract}
\vskip 2mm

We study character varieties of symmetric knots and their reductions $\mod p$.
We observe that the varieties present a different behaviour according to
whether the knots admit a free or periodic symmetry. 
\vskip 2mm

\noindent\emph{AMS classification: } Primary 57M25; Secondary 20C99; 57M50.

\vskip 2mm

\noindent\emph{Keywords:} Character varieties, hyperbolic knots, symmetries.

\end{abstract}

\section{Introduction}
\label{s:introduction}

Character varieties of $3$-manifold groups provide a useful tool in
understanding the geometric structures of manifolds and notably the presence of
essential surfaces. In this paper we wish to investigate 
$\mathrm{SL}_2$-character varieties of symmetric hyperbolic knots in order to 
pinpoint specific behaviours related to the presence of free or periodic 
symmetries. We will be mostly concerned with symmetries of odd prime order and 
we will concentrate our attention to a subvariety of the character variety,
i.e. the \emph{invariant subvariety} in the sense of algebraic geometry, which 
is pointwise fixed by the action of the symmetry (see 
Section~\ref{s:invariantch} for a precise definition of this action and of the 
invariant subvariety).

As already observed in \cite{HLM}, the excellent component of the character
variety containing the character of the holonomy representation is fixed
pointwise by the symmetry, since the symmetry can be chosen to act as a 
hyperbolic isometry of the complement of the knot. Hilden, Lozano, and 
Montesinos also observed that the invariant subvariety of a hyperbolic 
symmetric (more specifically, periodic) knot can be sometimes easier to 
determine than the whole variety. This follows from the fact that the 
invariant subvariety can be computed using the character variety of a 
two-component hyperbolic link. Such link is obtained as the quotient of the 
knot and the axis of its periodic symmetry by the action of the symmetry 
itself. Indeed, the link is sometimes much ``simpler" than the original knot, 
in the sense that its fundamental group has a smaller number of generators and 
relations, making the computation of its character variety feasible. This is, 
for instance, the case when the quotient link is a $2$-bridge link: Hilden, 
Lozano, and Montesinos studied precisely this situation and were able to 
recover a defining equation for the excellent components of several periodic
knots up to ten crossings.

In what follows we will be interested in the structure of the invariant
subvariety itself and we will consider not only knots admitting periodic
symmetries but also free symmetries. Our main result shows that the invariant
subvariety has in general a different structure according to whether the knot 
admits a free or periodic symmetry.

\begin{Theorem}
\label{thm:main}
If $K$ has a periodic symmetry of prime order $p\geq 3$, then $X(K)$ contains 
at least $(p-1)/2$ components that are curves and that are pointwise fixed by 
the symmetry.

On the other hand, for each prime $p\geq 3$, there is a knot $K_p$ with a free 
symmetry of order $p$ such that the number of components of the invariant 
character variety of $K_p$ is bounded, independently of $p$.
\end{Theorem}

The main observation here is that the invariant subvariety for a hyperbolic 
symmetric knot, or more precisely the Zariski-open set of its irreducible 
characters, can be seen as a subvariety of the character variety of a 
well-chosen two-component hyperbolic link, even when the symmetry is free.

To make the second part of our result more concrete, in Section~\ref{s:examples}
we study an infinite family of examples all arising from the two-component
$2$-bridge link $6^2_2$ in Rolfsen's notation (with $2$-bridge invariant
$10/3$). Our construction provides infinitely many knots with free symmetries
such that the number of irreducible components of the invariant subvarieties of
the knots is universally bounded. 

As for the rest of the character variety, that is the components that are
either invariant but not pointwise fixed, or equivariant, it is, in general,
much harder to determine. In Section~\ref{s:more} we study a family of knots 
having either free or periodic symmetry (with both situations occurring) whose 
character varieties contain equivariant components. This suggests that one 
cannot distinguish the nature of the symmetry of the knot by looking at this
part of the variety.

The invariant subvarieties of periodic knots over fields of positive 
characteristic exhibit another peculiar behaviour. It is well-known that for 
almost all odd primes $p$ the character variety of a finitely presented group
resembles the character variety over ${\mathbb C}$. For a finite set of primes,
though, the character variety over $p$ may differ from the one over ${\mathbb
C}$, in the sense that there may be ``jumps" either in the dimension of its
irreducible components or in their number. In this case we say that \emph{the
variety ramifies at $p$}. The character varieties of the knots
studied in \cite{PP} provide the first examples in which the dimension of a
well-defined subvariety of the character variety is larger for certain primes.
Here we give an infinite family of periodic knots for which the invariant
character variety ramifies at $p$, where $p$ is the order of the period. In
this case, the ramification means that the number of $1$-dimensional components
of the invariant subvariety decreases in characteristic $p$. This gives some 
more insight in the relationship between the geometry of a knot and the algebra 
of its character variety, namely the primes that ramify. 

The paper is organised as follows: Section~\ref{s:quotientlink} is purely 
topological and describes how one can construct any symmetric knot starting 
from a well-chosen two-component link. Section~\ref{s:chvar} provides basic 
facts on character varieties and establishes the setting in which we will work. 
In Section~\ref{s:invariantch} we introduce and study invariant character 
varieties of symmetric knots. The first part of Theorem~\ref{thm:main} on 
periodic knots is proved in Section~\ref{s:periodic} while in 
Section~\ref{s:free} we study properties of invariant character varieties of 
knots with free symmetries. The proof of Theorem~\ref{thm:main} is achieved in 
Section~\ref{s:examples}, where an infinite family of free periodic knots with 
the desired properties is constructed. In Section~\ref{s:more} we discuss the
non invariant part of the variety for a family of Montesinos knots: these have 
either a free or a periodic symmetry and their character varieties contain
equivariant components. Finally, in Section~\ref{s:modp} we describe how the 
character varieties of knots with period $p$ may ramify $\mod p$.

\section{Symmetric knots and two-component links}\label{s:quotientlink}

Let $K$ be a knot in ${\mathbf S}^3$ and let $\psi:({\mathbf
S}^3,K)\longrightarrow ({\mathbf S}^3,K)$ be a finite order diffeomorphism of
the pair which preserves the orientation of ${\mathbf S}^3$. 

\begin{definition}
If the group $\langle \psi\rangle$ acts freely we say that $\psi$ is a 
\emph{free symmetry of $K$}. If $\psi$ has a global fixed point then, according 
to the positive solution to Smith's conjecture \cite{MorganBass}, the 
fixed-point set of $\psi$ is an unknotted circle and two situations can arise: 
either the fixed-point set of $\psi$ is disjoint from $K$, and we say that 
$\psi$ is a \emph{periodic symmetry of $K$}, or it is not. In the latter case 
$\psi$ has order $2$, its fixed-point set meets $K$ in two points, and $\psi$ 
is called a \emph{strong inversion of $K$}. In all other cases $\psi$ is called 
a \emph{semi-periodic symmetry of $K$}.
\end{definition}

\begin{remark}
Note that if the order of $\psi$ is an odd prime, then $\psi$ can only be a
\emph{free} or \emph{periodic symmetry} of $K$.
\end{remark}
 
We start by recalling some well-known facts and a construction that will be
central in the paper.

Let $L=A\sqcup K_0$ be a hyperbolic two-component link in the $3$-sphere such 
that $A$ is the trivial knot. Let $n\ge2$ be an integer and assume that $n$ and 
the linking number of $A$ and $K_0$ are coprime. We can consider the $n$-fold 
cyclic cover $V={\mathbf D}^2\times{\mathbf S}^1\longrightarrow E(A)$ of the 
solid torus $E(A)$ which is the exterior of $A$ and contains $K_0$. The lift of 
$K_0$ in $V$ is a (connected) simple closed curve $C$.

Let $\mu,\lambda$ be a meridian-longitude system for $A$ on $\partial E(A)$ and
let $\tilde\mu,\tilde\lambda$ be its lift on $\partial V$. The slopes 
$\gamma_k=\tilde\mu+k\tilde\lambda$, for $k=0,\dots,n-1$, on $\partial V$ are 
equivariant by the action of the cyclic group ${\mathbb Z}/n{\mathbb Z}$ of 
deck transformations of the covering $V\longrightarrow E(A)$. More precisely,
since the group acts by translation in the direction $\tilde\mu$, $\gamma_k$ is
invariant if $k=0$ and has an orbit consisting of $d=n/(n\wedge k)$ disjoint 
images otherwise. Note that the manifold $V_k$ obtained by Dehn filling 
$V$ along $\gamma_k$ is ${\mathbf S}^3$. The action of the group of deck 
transformations ${\mathbb Z}/n{\mathbb Z}$ on $V$ extends to an action on 
$V_k$ which is free if $k\neq 0$ is prime with $n$ and has a circle of fixed 
points if $k=0$. For all other values of $k$, the action is semi-periodic, that 
is a proper subgroup of ${\mathbb Z}/n{\mathbb Z}$ of order $n/d$ acts with a 
circle of fixed points. 

For a fixed $k$, the image of $C$ in $V_k$ is a knot that we will denote by $K$ 
admitting a periodic or free symmetry of order $n$ according to whether $k=0$ 
or prime with $n$. For $n$ large enough, the resulting knot $K$ is hyperbolic
because of Thurston's hyperbolic Dehn surgery theorem 
\cite{ThurstonNotes}, e.g. \cite[App. B]{BoileauPorti}.

\begin{remark}\label{r:surgery}
Of course, the above construction can be carried out for arbitrary integer 
values of $k$. However, it is not restrictive to require the value of $k$ to be 
$\ge 0$ and $<n$. Indeed, assume that $k'=k+an$ where $0\le k<n$. The knot $K'$ 
resulting from $1/k'$ surgery along $V$ coincides with the knot $K$ obtained in
the same manner but starting from a different link $L'$ and choosing 
$\gamma_{k}=\tilde\mu+k\tilde\lambda$ as Dehn filling slope. The link $L'$ is 
obtained from $L$ by Dehn surgery of slope $1/a$ along $A$. 
\end{remark}

The following proposition shows that periodic and free-symmetric knots can 
always be obtained this way.

\begin{Proposition}\label{p:quotientlink}
Let $K$ be a hyperbolic knot admitting a free or periodic symmetry of order
$n$. Then there exists a two-component hyperbolic link 
$L=A\sqcup K_0\subset{\mathbf S}^3$ with $A$ the trivial knot, and an integer 
$0\le k<n$ such that the knot $K$ can be obtained by the above construction.
\end{Proposition}

\begin{proof}
The statement is obvious if the symmetry is periodic: in this case the link $L$
consists of the image $A$ of the axis of the symmetry and the image $K_0$ of
the knot $K$ in the quotient of ${\mathbf S}^3$ by the action of the symmetry.
Hyperbolicity of the link is a straightforward consequence of the hyperbolicity
of $K$ and the orbifold theorem.

If the symmetry is free, some extra work is necessary. The quotient of 
${\mathbf S}^3$ by the action of the free symmetry is a lens space containing a 
hyperbolic knot $K_0'$, image of $K$. 

Consider the cores of the two solid tori of a genus-$1$ Heegaard splitting for 
the lens space induced by an invariant genus-$1$ splitting of ${\mathbf S}^3$. 
Up to small isotopy one can assume that $K_0'$ misses one of them, say $\alpha$.
Note that the free homotopy class of $\alpha$ is non-trivial both in the lens
space and in the complement of $K_0'$. Observe, moreover, that the exterior of
$\alpha$ is a solid torus.

Let $\tilde \alpha\subset {\mathbf S}^3-K$ denote the lift of $\alpha$. If 
$K\sqcup \tilde\alpha$ is a hyperbolic link, then we are done by choosing
$L=K_0\sqcup A$ to be any link in ${\mathbf S}^3$ such that the exterior of $A$ 
coincides with the exterior of $\alpha$ and $(E(A),K_0)=(E(\alpha),K_0')$. In
other words, $K_0$ is the image of $K_0'$ in a chosen Dehn surgery along
$\alpha$ resulting in ${\mathbf S}^3$, and $A$ is the core of the surgery.

If $K\sqcup \tilde\alpha$ is not hyperbolic we will modify the choice of 
$\tilde\alpha$. For basic terminology in $3$-dimensional topology the reader is
referred to \cite{Hat} where the JSJ-decomposition is also discussed (for the
latter see also \cite{JS,J})

First of all, note that the link $K\sqcup \tilde\alpha$ is not split. This is a 
consequence of the equivariant sphere theorem and the fact that $\tilde\alpha$ 
is invariant, hence $E(K\sqcup \tilde\alpha) $ is irreducible and boundary 
irreducible. Indeed, note that a compressing disc for the link would have a 
regular neighbourhood whose boundary would be a splitting sphere. In addition 
$E( K\sqcup \tilde\alpha)$ is not Seifert fibred, because a Dehn filling on 
$\tilde \alpha$ yields $E(K)$, which is hyperbolic. Thus the only obstruction 
to hyperbolicity is that $E(K\sqcup \tilde\alpha)$ could be toroidal. Observe
that if $E(K\sqcup \tilde\alpha)$ is annular, for instance if $K$ and 
$\tilde\alpha$ are parallel, then either $E(K\sqcup \tilde\alpha)$ is toroidal 
or Seifert fibred.

Assume that its JSJ-decomposition is nontrivial and let $M$ be the piece of 
the splitting that is closest to $K$. Note that the JSJ-decomposition can be
chosen to be invariant by the action of the symmetry, so that, in particular, 
$M$ is invariant. %by the action of the symmetry. 
The boundary of $M$ consists of $T_0^2=\partial \mathcal N (K)$, some tori 
$T^2_1,\ldots, T^2_k$, $k\geq 0$ that do not separate $\tilde \alpha$ from $K$, 
and possibly a torus $T_{k+1}^2$ that separates $M$ from $\tilde \alpha$. We 
shall modify $\tilde \alpha$ so that $k=0$ and $T^2_ {k+1}=\partial \mathcal 
N(\tilde\alpha)$, which will yield hyperbolicity.

By hyperbolicity of $K$, for $i\geq 1$, each $T_i^2$ either bounds a solid 
torus in $E(K)$ or it is contained in a ball in $E(K)$. Notice that $T_{k+1}^2$ 
must bound a solid torus in $E(K)$, because $\tilde \alpha$ is not contained in 
a ball else the link $K\sqcup \tilde\alpha$ would be split. In addition, none 
of the $T^2_1,\ldots, T^2_k$ can bound a solid torus in $E(K)$, by 
nontriviality of the JSJ-decomposition. Note that, for each $1\le i\le k$, 
$\tilde \alpha$ must pass through the ball in $E(K)$ that contains $T_i^2$. 

%First we modify $\tilde \alpha$ so that $\partial T_{k+1}^2=\partial \mathcal 
%N(\tilde\alpha)$. Let $V$ be the solid torus bounded by $T_{k+1}^2$. Then 
%$\tilde \alpha\subset V$ and $V$ must be equivariant. In addition $V$ is not 
%knotted, because $\tilde \alpha$ is the trivial knot but also a satellite with 
%companion ${\mathbf S}^3\setminus V$. Then the modification consists in 
%replacing $\tilde\alpha$ by the core of $V$. This makes $\partial T_{k+1}^2$ 
%boundary parallel, and hence inessential.

%Finally, we 
We start by getting rid of the tori $T^2_1,\ldots, T^2_k$. Let 
$B^3_i\subset E(K)$ denote the $3$-ball containing $T_i^2$, for $i=1,\ldots,k$. Note that $T_i^2$ cuts out a knot exterior inside $B^3_i$ whose complement in 
the ball is a solid cylinder contained in $M$. As a consequence, in each ball 
there is a proper arc $\beta_i\subset B_i^3$ such that 
$N_i=B^3_i\setminus\mathcal N(\beta_i)$ is a knot exterior with boundary 
(parallel to) $T^2_i$. Moreover $N(\beta_i)$ contains the intersection of 
$\tilde\alpha$ with $B^3_i$; we stress again that $B^3_i$ does not meet $K$. 
This is schematically illustrated in Figure~\ref{f:surgery}. We now replace 
equivariantly each $N_i$ with a solid torus, that is the exterior of a trivial 
knot. Observe that on each $T_i^2$ there is a well-defined longitude-meridian 
system of curves corresponding to the longitude-meridian system of the knot 
exterior $N_i$. Replacing $N_i$ with a solid torus corresponds to performing a 
Dehn filling of $M$ along $T_i^2$ with slope the longitude. The $N_i$ that are 
permuted by the symmetry are exteriors of the same knot and the symmetry must 
preserve longitudes. This ensures that the filling on the $T_i^2$s can be 
carried out in an equivariant way. The effect of the surgery on a $T_i^2$ is 
shown again in Figure~\ref{f:surgery}. Note that by construction there is a 
degree-one map from $N_i$ to the added solid torus which is the identity on the 
boundary $T_i^2$: the surgery is obtained by pinching a Seifert surface for 
$N_i$ to a disc. Note that the only effect of this surgery is to modify 
$\tilde\alpha$ inside the $B^3_i$, in particular the resulting manifold is 
again ${\mathbf S}^3$. Moreover the surgery did not change $K$, because it just 
consisted in replacing the balls $B^3_i$ that are disjoint from $K$ again with 
balls. Denote by $\tilde\alpha'$ the image of $\tilde\alpha$ after the surgery.
Obviously $\tilde\alpha'$ is still invariant by construction. We want to show
that $\tilde\alpha'$ is a unknotted. For this observe that there is a 
degree-one map from $({\mathbf S}^3,\tilde\alpha)$ to $({\mathbf
S}^3,\tilde\alpha')$ which is the identity on the complement of the union of
the $N_i$ and coincides with the degree-one maps defined above on each $N_i$.
This map clearly induces a degree-one map from the exterior of $\tilde\alpha$
to that of $\tilde\alpha'$. The conclusion follows for $\tilde\alpha$ is the
trivial knot and its exterior cannot have a degree-one map on the exterior of a 
non-trivial knot.

Note that since the surgeries we performed are ``small", we cannot be sure that
the resulting link $K\sqcup \tilde\alpha'$ is hyperbolic yet. If $K\sqcup
\tilde\alpha'$ admits again essential tori that do not separate $K$ from
$\tilde\alpha'$, we repeat the procedure just seen. To be able to conclude we 
need to make sure that we will find an atoroidal link in a finite number of 
steps. This follows from the fact that at each step the simplicial volume of 
the link strictly decreases and from the structure of the 
set of hyperbolic volumes. One can also use the fact that at each step we have
degree-one maps from the old link to the new one and use the fact that there
are no infinite chains of such maps by \cite{Rong}.

%On 
%the other hand, this may change $\tilde\alpha$ to $\tilde\alpha'$, but since 
%every knot exterior has a degree-one map onto the solid torus, we find a 
%degree-one map from $E(\tilde\alpha)$ onto $E(\tilde\alpha')$, and since 
%$\tilde\alpha$ is unknotted, so is $\tilde\alpha'$.

By the above argument we can thus assume that for $K\sqcup \tilde\alpha$ $k=0$. 
We will modify $\tilde \alpha$ so that $\partial T_{k+1}^2=\partial \mathcal 
N(\tilde\alpha)$. Let $V$ be the solid torus bounded by $T_{k+1}^2$. Then
$\tilde \alpha\subset V$ and $V$ must be equivariant. In addition $V$ is not
knotted, because $\tilde \alpha$ is the trivial knot but also a satellite with
companion ${\mathbf S}^3\setminus V$. Then the modification consists in
replacing $\tilde\alpha$ by the core of $V$. This makes $\partial T_{k+1}^2$
boundary parallel, and hence inessential.

We can now quotient the resulting hyperbolic link by the symmetry to obtain a 
new $K_0'\sqcup \alpha'$ in a lens space. The link $L$ is then obtained as in 
the previous case.
\end{proof}

\begin{figure}[h]
\begin{center}
 {
  \includegraphics[height=4cm]{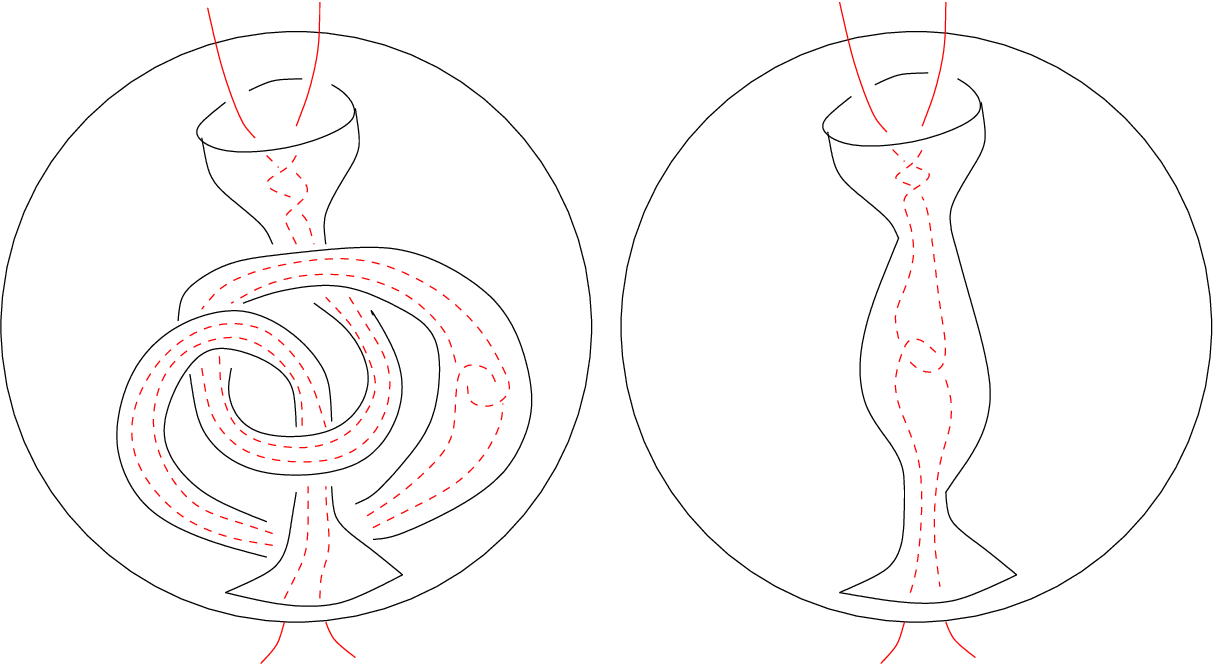}
 }
%\scalebox{0.5}
% {\input{involutions.pstex_t}}
\end{center}
\caption{A schematic picture of one of the balls $B^3_i$ with $\tilde\alpha$
passing through it before the surgery (on the left), and the effect of surgery 
(on the right).}
\label{f:surgery}
\end{figure}

Note that for a given $K$ the choice of $L$ is not unique. Indeed, links are
not determined by their complements, and there are infinitely many slopes on
the boundary of a solid torus such that performing Dehn filling along them
gives the $3$-sphere (see also Remark~\ref{r:surgery}).

\begin{remark}
Note that if $K$ admits a semi-periodic symmetry, then either all powers of the
symmetry that act as periods have the same fixed-point set or the union of 
their fixed-point sets consists of two circles forming a Hopf link. In the 
first situation a hyperbolic link $L$ can be constructed as in the case of 
periodic knots. In the second situation, one can construct $L$ by
choosing one of the two components of the Hopf link, but $L$ will not be
hyperbolic in general. Since we only consider symmetries of odd prime order in
the following, we are not going to analyse this situation further.
\end{remark}

\section{Character varieties}\label{s:chvar}

Let $G$ be a finitely presented group. Given a representation 
$\rho:G\longrightarrow \mathrm{SL}_2(\mathbb{C})$, its character is the map
$\chi_{\rho}:G\to \mathbb C$ defined by
$\chi_{\rho}(\gamma)=\operatorname{trace}(\rho(\gamma))$, $\forall\gamma\in G$.
The set of all characters is denoted by ${\mathbf X}(G)$.

Given an element $\gamma\in G$, we define the map
$$
\begin{array}{rcl}
     \tau_\gamma: {\mathbf X}(G) & \to & \mathbb C \\
              \chi & \mapsto & \chi(\gamma)
\end{array}.
$$

\begin{Proposition}[\cite{CullerShalen,AcunaMontesinos}]
\label{proposition:X(G)}
The set of characters ${\mathbf X}(G)$ is an affine algebraic set defined over
$\mathbb Z$, which embeds in $\mathbb C^N$ with coordinate functions
$(\tau_{\gamma_1}, \ldots, \tau_{\gamma_N})$ for some
$\gamma_1,\ldots,\gamma_N\in G$.
\end{Proposition}

The affine algebraic set ${\mathbf X}(G)$ is called the \emph{character 
variety} of $G$: it can be interpreted as the algebraic quotient of
the variety of representations of $G$ by the conjugacy action of 
$\mathrm{PSL}_2(\mathbb{C})=
\mathrm{SL}_2(\mathbb{C})/\mathcal{Z}(\mathrm{SL}_2(\mathbb{C}))$.

Note that the set $\{\gamma_1,\ldots,\gamma_N\}$ in the above proposition can
be chosen to contain a generating set of $G$. For $G$ the fundamental group of
a knot exterior, we will then assume that it always contains a representative
of the meridian.

A careful analysis of the arguments in \cite{AcunaMontesinos} shows that
Proposition~\ref{proposition:X(G)} still holds if $\mathbb C$ is replaced by
any algebraically closed field, provided that its characteristic is different
from $2$. Let $\mathbb F_p$ denote the field with $p$ elements and 
$\bar{\mathbb F}_p$ its algebraic closure. We have:

\begin{Proposition}[\cite{AcunaMontesinos}]
\label{proposition:X(G)Fp}
Let $p>2$ be an odd prime number. The set of characters ${\mathbf 
X}(G)_{\bar{\mathbb F}_p}$ associated to representations of $G$ over the field 
$\bar{\mathbb F}_p$ is an algebraic set which embeds in $\bar{\mathbb F}_p ^N$ 
with the same coordinate functions $(\tau_{\gamma_1}, \ldots, \tau_{\gamma_N})$ 
seen in Proposition~\ref{proposition:X(G)}. Moreover, ${\mathbf 
X}(G)_{\bar{\mathbb F}_p}$ is defined by the reductions mod $p$ of the 
polynomials over $\mathbb Z$ which define ${\mathbf X}(G)_\mathbb{C}$.
\end{Proposition}

One of the steps of the proof of Proposition~\ref{proposition:X(G)} in 
\cite{AcunaMontesinos} is the fact that, for any $g\in G$, $\tau_g\in \mathbb Z [\tau_{\gamma_1}, \ldots, \tau_{\gamma_N}]$. This yields:

\begin{remark}
\label{remark:inducedmap}
For any group homomorphism $\phi\colon G\to H$, the induced map
$\phi^*\colon {\mathbf X}(H)_\mathbb{C}\to {\mathbf X}(G)_\mathbb{C}$ is 
polynomial with coefficients in $\mathbb Z$, because its coordinates are 
obtained by writing $\tau_{\phi(\gamma_1)},\ldots, \tau_{\phi(\gamma_N)} $
as polynomials in the traces of the elements in $H$ provided by 
Proposition~\ref{proposition:X(G)}.

In addition, the reduction mod $p$ of 
$\phi^*\colon {\mathbf X}(H)_\mathbb{C}\to {\mathbf X}(G)_\mathbb{C}$ is
the induced map
${\mathbf X}(H)_{\bar{\mathbb F}_p}\to {\mathbf X}(G)_{\bar{\mathbb F}_p}$.
\end{remark}

Let ${\mathbb K}$ be an algebraically closed field of characteristic different
from $2$. A representation $\rho$ of $G$ in $\mathrm{SL}_2({\mathbb K})$ is 
called \emph{reducible} if there is a $1$-dimensional subspace of 
${\mathbb K}^2$ that is $\rho(G)$-invariant; otherwise $\rho$ is called 
\emph{irreducible}. The character of a representation $\rho$ is called 
\emph{reducible} (respectively \emph{irreducible}) if so is $\rho$. 

The set of reducible characters coincides with the set of characters of abelian
representations. Such set is Zariski closed and moreover is a union of 
irreducible components of ${\mathbf X}(G)$ that we will denote ${\mathbf 
X}_{{\mathrm ab}}(G)$ \cite{CullerShalen}.

Assume now that $G$ is the fundamental group of a link in the $3$-sphere with
$r$ components. In this case, ${\mathbf X}_{{\mathrm ab}}(G)$ is
an $r$-dimensional variety that coincides with the character variety of
 ${\mathbb Z}^r$, i.e. the homology of the link. In the 
case where $r=1$, that is the link is a knot, ${\mathbf X}_{{\mathrm ab}}(G)$ 
is a line parametrised by the trace of the meridian. When $r=2$, that is the 
link has two components, ${\mathbf X}_{{\mathrm ab}}(G)$ is parametrised by the 
traces $x,y$ of the two meridians and that, $z$, of their product subject to 
the equation $x^2+y^2+z^2-xyz-4=0$.

The subvariety of abelian characters is well-understood for the groups that we
will be considering. Hence, in the rest of the paper, we will only consider the 
irreducible components of ${\mathbf X}(G)$ that are not contained in the 
subvariety of abelian characters.

\begin{notation}\label{notation}
We will denote by $X(G)$ the Zariski closed set which is the union of of the 
irreducible components of ${\mathbf X}(G)$ that are not contained in the
subvariety of abelian characters. If $G$ is the fundamental group of a manifold
or orbifold $M$ we will write for short $X(M)$ instead of $X(G)$. Similarly if
$G$ is the fundamental group of the exterior of a link $L$ we shall write
$X(L)$ instead of $X(G)$. Notice that if $G$ is the fundamental group of a 
finite volume hyperbolic manifold then $X(G)$ is non empty for it contains the 
character of the hyperbolic holonomy.
\end{notation}

Assume now that $f$ is in $Aut(G)$. The automorphism $f$ induces an
action on both ${\mathbf X}(G)$ and $X(G)$ defined by $\chi\mapsto \chi\circ 
f$. This action only depends on the class of $f$ in $Out(G)$ since traces are
invariant by conjugacy. According to Remark~\ref{remark:inducedmap} applied to 
$G=H$ and $\phi=f$, we see that this action on the character varieties is
realised by an algebraic morphism defined over ${\mathbb Z}$.
%Moreover, the action on the character varieties is 
%realised by an algebraic morphism defined over ${\mathbb Z}$.
%, \textcolor{red}{see Remark~\ref{remark:inducedmap}}. 
It follows 
readily that the set of fixed points of the action is Zariski closed and itself 
defined over ${\mathbb Z}$. As a consequence, the defining relations of the 
variety of characters that are fixed by the action considered over a field of
characteristic $p$, an odd prime number, are just the reduction $\mod p$ of the 
given equations with integral coefficients.

\section{The character variety of $L$ and the invariant subvariety of $K$}
\label{s:invariantch}

In this section we define and study the invariant subvariety of $K$, where $K$
is a hyperbolic knot admitting a free or periodic symmetry of order an odd 
prime $p$. 
 
Let $\psi$ denote the symmetry of $K$ of order $p$ and let $L=K_0\sqcup A$ be 
the associated link as defined Section~\ref{s:quotientlink}. Denote by 
$E(K)/\psi$ the space of orbits of the action of $\psi$ on the exterior $E(K)$ 
of the knot $K$. The space $E(K)/\psi$ is either a manifold or an orbifold 
according to whether $\psi$ is free or periodic. Recall that $E(K)/\psi$ is 
obtained by a (possibly orbifold) Dehn filling on the component $A$ of the link 
$L$. We have
$$1\longrightarrow \pi_1(K)\longrightarrow
\pi_1(E(K)/\psi)\longrightarrow{\mathbb Z}/p{\mathbb Z}\longrightarrow 1$$
where $\pi_1(E(K)/\psi)$ denotes the orbifold fundamental group if $\psi$ is
periodic. We have that the sequence splits if and only if $\psi$ is periodic. 
Note that if $\psi$ is free then the quotient group ${\mathbb Z}/p{\mathbb Z}$ 
can also be seen as the fundamental group of the lens space quotient. In any 
case, we see that $\psi$ defines an element $\psi_*$ of the outer automorphism 
group of $\pi_1(K)$. Remark now that, since $E(K)/\psi$ is obtained by Dehn 
filling a component of $L$, the exterior $E(L)$ of the link $L$ is naturally 
embedded into $E(K)/\psi$. Let $\mu$ be an element of $\pi_1(E(K)/\psi)$ 
corresponding to the image of a meridian of $A$ via this natural inclusion: it 
maps to a generator of ${\mathbb Z}/p{\mathbb Z}$. Let $f\in Aut(\pi_1(K))$ be 
the automorphism of $\pi_1(K)$ induced by conjugacy by $\mu$. Note that $f$ is 
a representative of $\psi_*$. Thus the symmetry $\psi$ induces an action on the 
character variety $X(K)$ of the exterior of $K$ as defined in the previous 
section.

We have seen that the fixed-point set of this action is an algebraic subvariety 
of $X(K)$. We will denote by $X(K)^\psi$ the union of its irreducible 
components that are not contained in ${\mathbf X}_{{\mathrm ab}}(K)$. Note that
$X(K)^\psi$ is non empty for the character of the holonomy is fixed by the
action. Remark also that each irreducible component of $X(K)^\psi$ contains at 
least one irreducible character by definition. Indeed, each irreducible 
component of $X(K)^\psi$ contains a whole Zariski-open set of irreducible 
characters. We shall call $X(K)^\psi$ \emph{the invariant subvariety of $K$}.

Let us now consider how the different character varieties of $K$ and $L$ are
related.

It is straightforward to see that the character variety $X(E(K)/\psi)$ of the 
quotient of the exterior $E(K)$ of $K$ by the action of the symmetry injects
into the character variety $X(L)$ of the exterior of $L$. Indeed the
(orbifold) fundamental group of $E(K)/\psi$ is a quotient of the fundamental
group of $L$, induced by the Dehn filling along the $A$ component of $L$. 

On the other hand, there is a natural map from $X(E(K)/\psi)$ to the invariant
submanifold $X(K)^\psi$ of $K$, induced by restriction in the short exact 
sequence above.

Assume now that $\chi$ is a character in $X(K)^\psi$ associated to an 
irreducible representation $\rho$ of $K$. We will show that $\rho$ extends in a
unique way to a (necessarily irreducible) representation of $E(K)/\psi$ giving 
a character in $X(E(K)/\psi)$ (observe that here we only use that $p$ is odd). 
This proves that the above natural map is one-to-one and onto when restricted 
to the Zariski-open set of irreducible characters.

Note that if $\rho$ is a representation of $\pi_1(K)$ that extends to a
representation of $\pi_1(E(K)/\psi)$ then, necessarily, its character must be
fixed by the symmetry $\psi$, for the action of $\mu$ on $\pi_1(K)$ is by 
conjugacy and cannot change the character of a representation.

The idea is to extend $\rho$ to $\pi_1(E(K)/\psi)$ by defining $\rho(\mu)$ in 
such a way that the action of $\mu$ by conjugacy on the normal subgroup 
$\pi_1(K)$ coincides with the action of the automorphism $f$. We know that
$\chi=[\rho]=[\rho\circ f]$. Since $\rho$ is irreducible, 
$\mathrm{SL}_2({\mathbb C})$ acts transitively on the fibre of $\chi$ so that 
there exists an element $M\in \mathrm{SL}_2({\mathbb C})$ such that 
$\rho\circ f=M\rho M^{-1}$ \cite[Theorem 1.28]{LubotzkyMagid}. The element $M$ 
is well-defined, up to multiplication times $\pm 1$, i.e. up to an element in 
the centre of $\mathrm{SL}_2({\mathbb C})$. The fact that $\psi$ has odd 
order implies that there is a unique way to choose the sign and so that 
$\rho(\mu)=M$ is well-defined. Note that in some instances $\rho(\mu)$ can be 
the identity.

We have thus proved the following fact.  

\begin{Proposition}\label{prop:invsubvar}
Let $K$ be a hyperbolic knot admitting a symmetry $\psi$ of prime odd order. 
The restriction map from the character variety of $E(K)/\psi$ to the 
$\psi$-invariant subvariety of $K$ induces a bijection between the Zariski-open 
sets consisting of their irreducible characters.
\end{Proposition}

\begin{remark}
Proposition~\ref{prop:invsubvar} holds more generally for hyperbolic knots 
admitting either a free or a periodic symmetry of odd order and for character 
varieties over fields of positive odd characteristic.
\end{remark}

\section{Knots with periodic symmetries}\label{s:periodic}

Let $K$ be a hyperbolic knot admitting a periodic symmetry $\psi$ of odd prime 
order $p$. Let $L=A\sqcup K_0$ be the associated quotient link. Denote by 
$t_\mu$ the coordinate of the variety $X(L)$ corresponding to the trace of 
$\mu$. Proposition~\ref{prop:invsubvar} implies at once that $X(K)^\psi$ is
birationally equivalent to a subvariety $Z\cup Z_0$, where
$Z\subset X(L)\cap(\cup_{\ell=1}^{p-1}\{t_\mu=2\cos(2\ell \pi/p)\})$ and
$Z_0\subset X(L)\cap\{t_\mu=2\}$.

Note that since $p$ is odd, the set $\{2\cos(2\ell\pi/p)\mid 
\ell=1,\ldots,{p-1}\}$
equals
$\{-2\cos(\ell\pi/p)\mid \ell=1,\ldots,(p-1)/2\}$. In particular this 
includes a lift to $\mathrm{SL}_2(\mathbb{C})$ of the holonomy of $E(K)/\psi$, 
when $t_\mu = -2\cos(\pi/p)$; observe that this means that the image of the 
meridian is conjugate to 
$$-\left (\begin{matrix}
                         e^{i\frac{\pi}{p}} & 0 \\
                         0 &  e^{-i\frac{\pi}{p}}                               
\end{matrix}\right)
,
$$
a rotation of angle $\frac{2\pi}{p}$ that has order $p$
in $\mathrm{SL}_2(\mathbb{C})$.

\begin{Proposition}\label{prop:periodic}
The variety $Z$ contains at least $(p-1)/2$ irreducible curves $Z_1,\ldots, 
Z_{(p-1)/2}$, each of which contains at least one irreducible character. As a 
consequence, all these components are birationally equivalent to a subvariety 
$\tilde Z_1,\ldots, \tilde Z_{(p-1)/2}$ of $X(K)^\psi$.

Furthermore, the curves $\tilde Z_1,\ldots, \tilde Z_{(p-1)/2}$ are irreducible 
components of the whole $X(K)$, not only the invariant part.
\end{Proposition}

\begin{proof}
First of all, remark that the intersection of $X(L)$ with the hyperplane 
$\{t_\mu=-2\cos(\pi/p)\}$ contains the holonomy character $\chi_1$ of the 
hyperbolic orbifold structure of $E(K)/\psi$. In particular, a component of 
$X(E(K)/\psi)\cap\{t_\mu=-2\cos(\pi/p)\}$ is an irreducible curve $Z_1$ 
containing $\chi_1$, the so called excellent or distinguished component of 
$E(K)/\psi$. This is the curve that, viewed as a deformation space, allows to 
prove Thurston's hyperbolic Dehn filling theorem \cite{ThurstonNotes}, e.g. 
\cite[App. B]{BoileauPorti}. 

The character $\chi_1$ takes values in a number field $\mathbf K$ containing 
the subfield $\mathbb Q(\cos\frac{\pi}{p})$ of degree $\frac{p-1}2$. The Galois 
conjugates of $\chi_1$ are contained in 
$X(L)\cap\{t_\mu=-2\cos(\ell\pi/p)\})$ for some $\ell=1,\ldots,(p-1)/2$.
As $\{-2\cos(\ell\pi/p)\mid \ell=1,\ldots,(p-1)/2\}$
is precisely the set of Galois conjugates of $t_\mu(\chi_1)$,
this yields the $\frac{p-1}2$ components defined by 
$t_\mu=-2\cos(\ell\pi/p)$, 
$\ell=1,\ldots,\frac{p-1}2$ (though the number of conjugates may be larger, 
depending on the degree of the number field $\mathbf K$).

To prove the assertion that these curves are irreducible components of $X(K)$, 
notice that the restriction $\chi_1\vert_{E(K)} $ is the holonomy of the 
hyperbolic structure of $E(K)$. Therefore, by Calabi-Weil rigidity, the Zariski 
tangent space of $\tilde Z_1$ at $\chi_1\vert_{E(K)} $ is one dimensional. 
This space is isomorphic to the first cohomology group of $\pi_1(E(K))$ with 
coefficients in the Lie algebra $\mathfrak{sl}(2, \mathbb C)$ twisted by the 
adjoint of the holonomy, cf.~\cite{LubotzkyMagid,Weil}. We aim to show that 
this cohomology group does not change under Galois conjugation of the 
representation. This cohomology group can be computed as $Z^1/B^1$, where 
$Z^1$ is the space of crossed morphisms, or maps 
$f\colon \pi_1(E(K))\to \mathfrak{sl}(2, \mathbb C)$ satisfying 
$f(\gamma_1\gamma_2)= 
f(\gamma_1)+ \operatorname{Ad}_{\operatorname{hol}(\gamma_1)} f(\gamma_2)$,  
$\forall\gamma_1,\gamma_2\in \pi_1(E(K))$, and  $B^1$ denotes the subspace of 
inner crossed morphisms, or maps 
$f_a(\gamma)= a-\operatorname{Ad}_{\operatorname{hol}(\gamma)}(a)$, 
$\forall\gamma\in \pi_1(E(K))$ and for some $a\in \mathfrak{sl}(2, \mathbb C)$,
cf.~ \cite[\S II.5]{Brown}. Let us check that the dimension of $Z^1$ and $B^1$ 
does not change under Galois conjugation. Firstly, a crossed morphism is 
determined by the image of a generating set for $ \pi_1(E(K) )$ of cardinality 
$k$, thus we view $Z^1\subset  \mathfrak{sl}(2, \mathbb C)^k$. In addition, $k$ 
elements in $ \mathfrak{sl}(2, \mathbb C)$ determine a crossed morphism iff
they satisfy linear relations given by a presentation of $ \pi_1(E(K) )$. More 
precisely, $Z^1$ is isomorphic to the kernel of a linear map 
$ \mathfrak{sl}(2, \mathbb C)^k\to  \mathfrak{sl}(2, \mathbb C)^r $, where $r$
is the number of relaions, with coefficients that are $\mathbb{Z}$-polynomials 
on the holonomy of the generators. Therefore its dimension does not change 
under Galois conjugation. Secondly, $B^1$ can be seen as the image of the 
linear map from 
$ \mathfrak{sl}(2, \mathbb C)$ to $Z^1\subset   \mathfrak{sl}(2, \mathbb C)^k$
that maps each $a\in \mathfrak{sl}(2, \mathbb C)$ to the crossed morphism 
$\gamma\mapsto a-\operatorname{Ad}_{\operatorname{hol}(\gamma)}(a)$.
By considering  a $\mathbb C$-basis for  $ \mathfrak{sl}(2, \mathbb C)$
in $ \mathfrak{sl}(2, \mathbb Q)$, the dimension of $B^1$ is independent of the Galois conjugate. Thus the Zariski tangent space of the curves we are looking 
at is one dimensional, which gives an upper bound on the dimension that 
establishes the final claim.
\end{proof}

Remark that $X(K)^\psi$ may contain other components than the ones described
above. In particular, if $K_0$ is itself hyperbolic, there is at least one
extra component whose characters correspond to representations that map $\mu$ 
to the trivial element, that is the lift of the excellent component of 
$E(K_0)$.  
 
\begin{Corollary}
Let $K$ be a hyperbolic knot which is periodic of prime order $p\neq 2$. Then 
$X(K)$ contains at least $\frac{p-1}{2}$ irreducible components which are 
curves. 

In addition there is an extra irreducible component when $K_0$ itself is 
hyperbolic.
\end{Corollary}

\begin{remark}
By considering the abelianisation ${\mathbb Z}\times{\mathbb Z}/p{\mathbb Z}$ 
of the fundamental group of the orbifold $E(K)/\psi$, it is not difficult to 
prove that ${\mathbf X}_{{\mathrm ab}}(E(K)/\psi)$ consists of $(p+1)/2$ lines.
On the other hand, the abelianisation of the fundamental group of the exterior
of $K$ consists in a unique line which is fixed pointwise by the action induced
by $\psi$ on ${\mathbf X}(K)$. It follows that, in general, the fixed 
subvariety of the whole character variety of $K$ is not birationally equivalent 
to the whole character variety of the orbifold. For this reason we have 
restricted our attention to $X(K)^\psi$.
\end{remark}

\section{Knots with free symmetries}\label{s:free}

Let $K$ be a hyperbolic knot admitting a free symmetry $\psi$ of odd prime
order $p$. Let $L=A\sqcup K_0$ be the associated link as defined in
Section~\ref{s:quotientlink} (see in particular
Proposition~\ref{p:quotientlink}).

In this case, the irreducible characters of $X(K)^\psi$ are mapped inside the
subvariety of $X(L)$ obtained by intersection with the hypersurface defined by
the condition that its characters correspond to representations that send 
$\tilde\mu+k\tilde\lambda$ to the trivial element, where $0<k<p$ is the integer
coprime with $p$ prescribed by Proposition~\ref{p:quotientlink}.

Note that in $\pi_1(E(K)/\psi)$ one has
$\tilde\mu+k\tilde\lambda=p\mu+k\lambda$. We write:
\begin{equation}
\label{eqn:Apoly}
\rho(\mu)=\begin{pmatrix}
     m_A & * \\
      0 & m_A^{-1}
    \end{pmatrix} \quad \mbox{and}\quad 
\rho(\lambda)=\begin{pmatrix}
     l_A & * \\
      0 & l_A^{-1}
    \end{pmatrix}
\end{equation}
Thus the representations of $E(K)/\psi$ must satisfy $m_A^pl_A^k=1$. This 
provides a motivation to look at the restriction to the peripheral subgroup 
$\pi_1(\partial \mathcal{N}(A))$ generated by $\mu$ and $\lambda$:
\begin{equation}
 \label{eqn:restriction}
  \mathrm{res}: X(L)\to  {\mathbf X}(\partial \mathcal{N}(A)).
\end{equation}
When this restriction has finiteness properties, we are able to find uniform 
bounds on the number of components of $X(K)^{\psi}$:

\begin{Proposition}
\label{prop:bound}
Assume that \eqref{eqn:restriction} is a dominant morphism and that the
dimension of $X(L)$ is at most $2$. Then there is a constant $C$ depending only 
on $X(L)$ such that the number of components of $X(K)^{\psi}$ is $\leq C$.
\end{Proposition}

Here by dominant morphism we mean a morphism whose restriction to each
irreducible component of $X(L)$ is dominant.

Notice that the components of $X(L)$ have dimension at least 
two \cite{ThurstonNotes}, the hypothesis in Proposition~\ref{prop:bound} 
implies in particular that they are always surfaces. We give in the next 
section an example of a link for which \eqref{eqn:restriction} is a dominant 
morphism. As a consequence we have:

\begin{Corollary}\label{c:freebound}
There exists a sequence of hyperbolic knots $K_p$ parametrised by infinitely 
many prime numbers $p$ such that $K_p$ has a free symmetry $\psi$ of order $p$ 
but $X(K_p)^\psi$ is bounded, uniformly on $p$.
\end{Corollary}

\begin{proof}[Proof of the proposition] 
Since $p$ and $k$ are coprime, there exist $q$ and $h$ such that the elements 
$p\mu+k\lambda$ and $q\mu+h\lambda$ generate the fundamental group of
$\partial \mathcal{N}(A)$. The character variety ${\mathbf X}(\partial 
\mathcal{N}(A))$ is a surface in $\mathbb{C}^3$ with coordinates 
$x=tr(p\mu+k\lambda)$, $y=tr(q\mu+h\lambda)$, and 
$z=tr((p+q)\mu+(k+h)\lambda)$, defined by the equation $x^2+y^2+z^2-xyz-4=0$. 
The equations $x=2$ and $y=z$ determine a line $D$ contained in the surface   
${\mathbf X}(\partial \mathcal{N}(A))$ which corresponds to the subvariety of 
characters of representations that are trivial on $p\mu+k\lambda$.   

To count the components of $X(K)^\psi$ it is enough to count the components of 
$\mathrm{res}^{-1}(D)$. The map $\mathrm{res}$ being a dominant morphism, there 
is a Zariski open subset of each irreducible component of $X(L)$ on which
the map is finite to one. As a consequence, by dimensional reasons, there is a 
finite number $N$ of curves in $X(L)$ which are mapped to points of 
${\mathbf X}(\partial \mathcal{N}(A))$. It follows that the number of 
irreducible components $\mathrm{res}^{-1}(D)$ is bounded above by $d+N$ where 
$d$ is the cardinality of the generic fibre of $\mathrm{res}$.
\end{proof}

\section{A family of examples}\label{s:examples}

Consider the two-component $2$-bridge link $6^2_2$ pictured in
Figure~\ref{f:link}.

\begin{figure}[h]
\begin{center}
 {
 \psfrag{K0}{$K_0$}
 \psfrag{A}{$A$}
 \psfrag{x}{$\mu$}
 \psfrag{y}{$\nu$}
  \includegraphics[height=4cm]{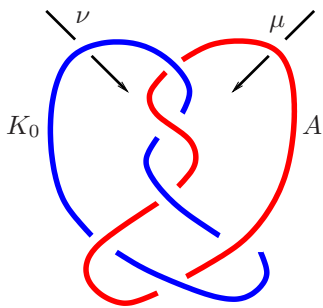}
 }
%\scalebox{0.5}
% {\input{involutions.pstex_t}}
\end{center}
\caption{The $2$-bridge link $6^2_2$ and the generators of its fundamental
group.}
\label{f:link}
\end{figure}

For each prime $p>4$ and each $0\le k<p$ one can construct a symmetric knot $K$
as described in Section~\ref{s:quotientlink}. Since the absolute value of the 
linking number of the two components of $L$ is $3$, the construction does not 
give a knot for $p=3$, which must thus be excluded.

Using Wirtinger's method one can compute a presentation of its fundamental
group:
$$\langle \mu,\nu \mid
\mu(\nu\mu^{-1}\nu\mu\nu^{-1}\mu\nu\mu^{-1}\nu)=
(\nu\mu^{-1}\nu\mu\nu^{-1}\mu\nu\mu^{-1}\nu)\mu \rangle$$
where the generators $\mu$ and $\nu$ are shown in Figure~\ref{f:link}.
Having chosen the meridian $\mu$, the corresponding longitude is  
$\lambda=\nu\mu^{-1}\nu\mu\nu^{-1}\mu\nu\mu^{-1}\nu$.

An involved but elementary computation gives the following defining equation
for $X(L)$
$$(\alpha\beta\gamma-\alpha^2-\beta^2-\gamma^2+4)
(-\gamma^4+\alpha\beta\gamma^3-(\alpha^2+\beta^2-3)\gamma^2+\alpha\beta\gamma-1)
$$
where $\alpha$, $\beta$, and $\gamma$ represent the traces of $\mu$, $\nu$, and
$\mu^{-1}\nu$ respectively. The equation can also be found in \cite{Harada}.

Note that the variety consists of two irreducible components, the first one
being that of the abelian characters.

A similar computation gives an expression for the trace of $\lambda$ in terms 
of $\alpha$, $\beta$, and $\gamma$: 
$${\mathrm tr}(\lambda)=
(\alpha\beta-(\alpha^2+\beta^2+\gamma^2-\alpha\beta\gamma-3)\gamma)
(\beta\gamma-\alpha)-(\alpha\gamma-\beta)$$

We want to understand the generic fibre of the restriction map $\mathrm{res}:
X(L)\to {\mathbf X}(\partial \mathcal{N}(A))$, where
$X(L)$ is a surface contained in ${\mathbb C}^3$ with
coordinates $\alpha$, $\beta$, and $\gamma$ and ${\mathbf X}(\partial 
\mathcal{N}(A))$ is also a surface contained in ${\mathbb C}^3$ but with 
coordinates ${\mathrm tr}(\mu)$, ${\mathrm tr}(\lambda)$, and 
${\mathrm tr}(\mu\lambda)$. For each fixed point $({\mathrm tr}(\mu), 
{\mathrm tr}(\lambda), {\mathrm tr}(\mu\lambda))$ in ${\mathbf 
X}(\partial \mathcal{N}(A))$, the fibre of $\mathrm{res}$
consists of the points $(\alpha,\beta,\gamma)$ which satisfy

$$
\begin{cases}
\alpha={\mathrm tr}(\mu) & \\
(\alpha\beta-(\alpha^2+\beta^2+\gamma^2-\alpha\beta\gamma-3)\gamma)
(\beta\gamma-\alpha)-(\alpha\gamma-\beta)={\mathrm tr}(\lambda) & \\
-\gamma^4+\alpha\beta\gamma^3-(\alpha^2+\beta^2-3)\gamma^2+\alpha\beta\gamma-1
=0 & \\
\end{cases}
$$ 
Once $\alpha$ is replaced by its value ${\mathrm tr}(\mu)$, the points we are 
interested in correspond to the intersection of two curves in ${\mathbb C}^2$ 
with coordinates $\beta,\gamma$. We see immediately that, for generic values of 
${\mathrm tr}(\lambda)$, each point of ${\mathbf X}(\partial \mathcal{N}(A))$ 
is the image of at most a finite number of points in $X(L)$ and such finite 
number is bounded above by the product of the degrees of the two polynomials in 
$\beta$ and $\gamma$, i.e. $20$.

This shows that Proposition~\ref{prop:bound} applies to this link and
Corollary~\ref{c:freebound} holds.

\section{Non invariant examples}
\label{s:more}

Here we construct examples of knots with free or periodic symmetries with 
characters that are not invariant. Firstly, we describe a component of the set 
of characters fixed by the symmetry that is strictly contained in an 
irreducible component (i.e.~fixed and non-fixed characters lie in the same 
irreducible component). Secondly, we show an example where some pairwise 
disjoint irreducible components of the variety of characters are permuted  
(in particular no character in the irreducible component is fixed).

We consider a special family of Montesinos knots, though our considerations can 
be easily adapted to other Montesinos knots. We use the description of 
symmetries of Montesinos links in \cite{BoileauZimmermann}. Following the 
notation there, consider the Montesinos link
$$
K=M(e,\frac{1}{\alpha},\overset{(q)}\ldots,\frac 1\alpha)\, ,
$$
with $q\geq 4$ and $\alpha\geq 5$. Since we are only interested in knots we 
require moreover that $\alpha$ and $q+e$ are odd. To describe a finite order 
diffeomorphism, view ${\mathbf S}^3$ as the join of two circles (i.e. the 
components of the Hopf link), and describe it as the join action obtained by 
composing rotations around each circle. 
By \cite[\S 3, Figure~ 5]{BoileauZimmermann}, $K$ can be arranged so that it is 
invariant by a diffeomorphism $\psi$ that is the composition of a rotation of 
angle $\frac{2\pi}q$ around one of the circles, and $\frac{\pi }q e$ around 
the other circle. In particular, $\psi$ has order $q$ if $e$ is even, and order 
$2 q$ if $e$ odd. In addition, $\psi$ is periodic if $e\in 2q\mathbb{Z}$, and 
$\psi$ is free if $e$ is even, and $e/2$ and $q$ are coprime. Note that we are
mainly concerned with symmetries of odd prime order, so in our case we can
choose $q$ to be an odd prime number and $e$ must be even.
%; observe that in this case $\alpha$ must be odd, too.

To construct representations of $E(K)$, use the orbifold Seifert fibration
induced by the double branched cover of $K$, as in Montesinos's original
description. Namely, consider $\mathcal O^3$ the orbifold with 
underlying space ${\mathbf S}^3$, branching locus $K$, and branching order $2$, 
so that $\pi_1(\mathcal O^3)=\pi_1(E(K))/\langle \mu^2\rangle$, where $\mu$ 
denotes a meridian. Since $K$ is a Montesinos knot, $\mathcal O^3$ is 
orbifold-Seifert fibred, with basis a Coxeter hyperbolic orbifold 
$\mathcal{B}^2$. Namely $\mathcal{B}^2=\mathbf{H}^2/\pi_1( \mathcal{B}^2) $, 
where $\pi_1( \mathcal{B}^2)$ is viewed as the Coxeter group corresponding to 
a hyperbolic polygon $P\subset \mathbf H^2$ with $q$ edges and angles 
$\pi/\alpha$. In particular $ \pi_1( \mathcal{B}^2)$ is generated by 
reflections $r_i$ along the edges of $P$, and it admits a presentation
$$
\pi_1( \mathcal{B}^2)=\langle r_1,\ldots,r_q \mid r_i^2= 
(r_i r_{i+1})^\alpha=1\rangle .
$$
The hyperbolic structure of $P$ is not unique: the angles being $\pi/\alpha$, 
the length of the edges of $P$ can be deformed in a $(q-3)$-dimensional space  
\cite{P1}, this yields a $(q-3)$-dimensional space of representations of 
$\Gamma$ into the isometry group of $\mathbf H^2$, which is in fact the 
Teichm\"uller space of the Fuchsian orbifold $\pi_1( \mathcal{B}^2)$,
that we denote by $\mathcal T$. Any isometry of $\mathbf H^2$ (orientable or 
not) extends uniquely to an orientable isometry of $\mathbf H^3$, in particular 
reflections along geodesics extend to rotations of angle $\pi$. Thus we get a 
space of characters of $\pi_1(\mathcal{B}^2)$ into $\mathrm{PSL}_2(\mathbb C)$. 
Let $T$ denote the component of the variety of 
$\operatorname{PSL}_2(\mathbb C)$-characters of $\pi_1( \mathcal{B}^2)$ that 
contains the holonomies of the Teichm\" uller space $\mathcal T$. Since we have 
a surjection $\pi_1({\mathbf S}^3-K)\to \pi_1(\mathcal{B}^2) $, this yields a 
family of representations of $\pi_1(E(K))$ to $\mathrm{PSL}_2(\mathbf C)$, that 
lift to $\mathrm{SL}_2(\mathbf C)$ by \cite{AcunaMontesinos}.

\begin{Lemma}
\label{lemma:liftedcomponent}
There is a component $Z$ of $X(E(L))$ of dimension $q-3$ consisting of lifts of 
representations induced by $T$.
\end{Lemma}

\begin{proof}
The Teichm\"uller space of $\mathcal B^2$ has real dimension $q-3$, as this is 
the dimension of the space of lengths of $P\subset \mathbf{H}^2$. If we view 
$P\subset \mathbf{H}^3$, and we allow $P$ not to be in a plane anymore, 
instead of lengths, to each edge we associate complex lengths (the real part is 
the usual length and the imaginary part a twist parameter). The very same 
argument as in \cite{P1} tells that the complex dimension of the space of 
possible complex lengths is $q-3$, and this yields complex dimension $q-3$ for 
$T$. 

Each representation in $T$ induces a representation of $\pi_1(E(K))$ to 
$\mathrm{PSL}_2(\mathbb C)$, and it lifts to precisely two representations in 
$\mathrm{SL}_2(\mathbb C)$, one for each choice of sign for the meridian. Thus 
we obtain at most two components in $\mathbf X(E(K))$ and we choose an 
irreducible component $Z$ that contains one of them. We claim that all 
characters in $Z$ are lifts of characters induced by representations of 
$\mathcal{B}^2$. Indeed, when a generic character in $T$ is viewed in 
$\mathbf X(E(K))$, the trace of the meridian remains constant in a (usual) 
neighborhood \cite[Theorem 2.10 and Lemma~3.6]{P2}, so it remains constant in 
$Z$. Hence all characters in $Z$ come from representations of $\mathcal O^3$ 
that, by irreducibility, must map the fibre of the fibration to a trivial 
element in $\mathrm{PSL}_2(\mathbb C)$, and so they come from a representation 
of $\mathcal{B}^2$.
\end{proof}

\begin{Proposition}
The component $Z$ is preserved by $\psi$ and contains characters that are fixed by $\psi$ and characters that are not.
Hence $\emptyset\neq Z^\psi\neq Z$.
\end{Proposition}

\begin{proof}
We follow the discussion in the previous paragraphs. A polygon in $\mathbf H^2$ 
is invariant by rotations iff it is regular, thus with given angles 
$\pi/\alpha$, the polygon  $P\subset \mathbf H^2$ is fixed iff all edge lengths 
are the same. Therefore, when we consider the corresponding characters in $Z$, 
we find fixed and non-fixed characters.
\end{proof}

For the next example, we consider components that generalise $Z$. Fix natural 
numbers $0 <\epsilon_1,\ldots,\epsilon_q< \alpha/2$, so that there exists a 
hyperbolic polygon $P_{\epsilon_1,\ldots,\epsilon_q}\subset \mathbf H^2$ with 
angles $\pi\epsilon_1/\alpha,\ldots,\pi\epsilon_q/\alpha$ (it exists by the 
bounds on the $\epsilon_i$ and Gauss-Bonnet). This polygon does not need to 
span a Coxeter group, but still it can be used to construct representations of 
$\pi_1(\mathcal B^2)$ to the isometry group of $\mathbf H^2$, just map $r_i$ to 
a reflection on the $i$-th edge, so that $r_i r_{i+1}$ is mapped to a rotation 
of angle $\pi\epsilon_i/\alpha$ instead of angle $\pi/\alpha$. Now the previous 
construction for $\epsilon_i=1$ applies verbatim and in this way we obtain 
components $Z_{\epsilon_1,\ldots,\epsilon_q}$ of $X(E(L))$ of dimension $q-3$. 
In particular $Z_{1,\ldots,1}= Z$. Moreover, for different values of the 
$\epsilon_i$ the components are disjoint, as the trace of the element of 
$\pi_1(E(K)$ mapped to $r_ir_{i+1}$ is $\pm 2\cos(\pi\epsilon_i/\alpha)$. As 
the symmetry $\psi$ cyclically permutes the rational tangles of $L$, it acts on 
the indices $\{\epsilon_1,\ldots,\epsilon_q \}$ by cyclic permutation. 
Summarising:

\begin{Proposition}
The components $Z_{\epsilon_1,\ldots,\epsilon_q}$ of $X(E(L))$ are pairwise 
disjoint for different values of $0 <\epsilon_1,\ldots,\epsilon_q< \alpha/2$,
$\epsilon_i\in\mathbb{N}$. The symmetry $\psi$ maps 
$Z_{\epsilon_1,\ldots,\epsilon_q}$ to 
$Z_{\epsilon_2,\ldots,\epsilon_q,\epsilon_1}$.
\end{Proposition}

\begin{remark}
For certain symmetric Montesinos knots, one can find other components of the 
character variety that are not contained in the invariant subvariety. For
instance, according to \cite[Theorem 1 \& Remark 1]{Li}, the knot 
$K=M(e,\frac{1}{\alpha}, \frac{ - 1}{\alpha}, \overset{( 2q)}\ldots,\frac  
1\alpha, \frac{ - 1}{\alpha})$, with $\alpha$ and $e$ odd, admits a degree-one 
map on $K_s=M(e,\frac{1}{\alpha}, \frac{ - 1}{\alpha},\overset{(2s)}\ldots,
\frac 1\alpha, \frac{ - 1}{\alpha})$, provided $3\le s\le q$. As a consequence, 
the character variety of $K$ contains that of $K_s$. Assume $q$ is an odd 
prime: $K$ admits a symmetry of order $q$ either periodic if $e$ is a multiple 
of $q$ or free otherwise. Such symmetry does not descend to a symmetry of $K_s$ 
if $s<q$. As a consequence, the characters of $K$ induced by irreducible 
representations of $K_s$ cannot be fixed by the action of the symmetry and must 
lie outside the invariant subvariety. In addition, the degree-one map preserves 
meridians 
%(it can be viewed as a cancellation of tangles). 
(it sends tangles to tangles by either preserving or reversing their
orientation). Thus, the same argument as in Lemma~\ref{lemma:liftedcomponent} 
yields that the Teichm\"uller component of $K_s$ induces genuine irreducible 
components of $\mathbf{X}(E(K))$ of dimension $2s-3$, which are permuted by 
the symmetry.
\end{remark}

\section{Invariant character varieties over fields of positive characteristic}
\label{s:modp}

Let $L=K_0\sqcup A$ be a hyperbolic link with two components such that $A$ is 
trivial. Assume that $\mathrm{lk}(K_0,A)\neq0$. For each odd prime number $p$ 
that does not divide the linking number $\mathrm{lk}(K_0,A)$, the knot $K_0$ 
lifts to a knot $K_p$ in the $p$-fold cyclic cover of ${\mathbf S}^3$ branched 
along $A$.

By construction (see Section~\ref{s:quotientlink}), $K_p$ admits a periodic 
symmetry $\psi$ of order $p$, and the invariant subvariety $X(K_p)^\psi$ 
contains at least $(p-1)/2$ irreducible components of dimension $1$. These 
components of $X(K_p)^\psi$ are constructed in Proposition~\ref{prop:periodic} 
as the intersection of the character variety $X(L)$ with a family of $(p-1)/2$ 
parallel hyperplanes. The union of these parallel hyperplanes corresponds to a 
hypersurface which is the vanishing locus of the minimal polynomial for 
$2\cos(2\pi/p)$ in the variable ${\mathrm tr}(\mu)$. Such polynomial can be 
easily computed from the $p$th cyclotomic polynomial and is defined over 
${\mathbb Z}$.  

The characters of $X(K_p)^\psi$ correspond to representations of the orbifold
$E(K_p)/\psi$. Note that $X(E(K_p)/\psi)$ may have further components besides
those provided by Proposition~\ref{prop:periodic}, since the orbifold may admit 
irreducible representations that are trivial on $\mu$. These irreducible 
representations correspond to characters for which ${\mathrm tr}(\mu)=2$. In 
any case, $X(E(K_p)/\psi)$ contains at least $(p-1)/2$ components of dimension 
$1$.

If we consider the character variety of $E(K_p)/\psi$ in characteristic $p$, we
have that, since the only elements of order $p$ are parabolic, the entire 
character variety must be contained in the hyperplane defined by 
${\mathrm tr}(\mu)=2$. We note that if $p$ is not a ramified prime for
$E(K_p)/\psi$, then it must contain as many $1$-dimensional irreducible
components as the one over ${\mathbb C}$, that is at least $(p-1)/2$.

Let us now turn our attention to the subvariety of $X(L)$ which consists in 
the intersection of $X(L)$ with the hyperplane ${\mathrm tr}(\mu)=2$. We remark
that it is non-empty since it must contain the character of the holonomy
representation of $L$. We are interested in its irreducible components of
dimension $1$. These are in finite number, say $N$, depending on $L$ only, and
constitute an affine variety of dimension $1$ that we shall denote $Y$.

Standard arguments of algebraic geometry show that for almost all (odd) primes
$q$, the character variety $X(L)$ as well as its subvariety $Y$ have the same
properties (like dimension and irreducible components) over an algebraically 
closed field of characteristic $q$ they have over the complex numbers. 
This follows basically from the fact that the dimension of an affine variety 
(that is, the maximal dimension of its irreducible components) and its 
irreducible components can be computed algorithmically (see, for instance, 
\cite[Chapter 9]{CLoS} for the dimension, and \cite[page 209]{CLoS} for the 
decomposition into irreducible components). In our specific situation we have 
that $Y$ has $N$ irreducible components.

\begin{Proposition}\label{p:prime}
For infinitely many periodic knots $K_p$ as above, the character variety
$X(K_p)$ ramifies at $p$.
\end{Proposition}

\begin{proof}
We start by considering the invariant variety $X(K_p)^\psi$ and show that this
variety ramifies at $p$ if $p$ is large enough. Indeed, if this were not the 
case, the above discussion implies that the number of irreducible curves of 
$X(K_p)^\psi$ should be at least $(p-1)/2$ on one hand and at most $N$ on the
other. It follows readily that $X(K_p)^\psi$ ramifies at $p$.

Now, since $(p-1)/2$ curves of the invariant variety $X(K_p)^\psi$ are also
irreducible components of $X(K_p)$ and since $X(K_p)^\psi$ is defined over
${\mathbb Z}$, the character variety of $K_p$ ramifies at $p$, too.  
\end{proof}

\begin{remark}
The polynomial equations defined over ${\mathbb Z}$ of the character variety of
the orbifold $E(K_p)/\psi$ generate a non radical ideal when considered $\mod
p$, since the minimal polynomial of $2\cos (\pi/p)$ is not reduced when
considered $\mod p$.
\end{remark}

\paragraph{Acknowledgement}
The authors are thankful to Mar\'{\i}a Teresa Lozano for valuable discussions.

\begin{footnotesize}

%}

\bibliographystyle{plain}

%  \bibliography{symmetry}

\begin{thebibliography}{10}

\bibitem{BoileauPorti}
Michel Boileau and Joan Porti.
\newblock Geometrization of 3-orbifolds of cyclic type.
\newblock {\em Ast\'erisque}, 272:208, 2001.
\newblock Appendix A by Michael Heusener and Porti.

\bibitem{BoileauZimmermann}
Michel Boileau and Bruno Zimmermann.
\newblock Symmetries of nonelliptic {M}ontesinos links.
\newblock {\em Math. Ann.}, 277(3):563--584, 1987.

\bibitem{Brown}
Kenneth~S. Brown.
\newblock {\em Cohomology of groups}, volume~87 of {\em Graduate Texts in
  Mathematics}.
\newblock Springer-Verlag, New York, 1994.
\newblock Corrected reprint of the 1982 original.

\bibitem{CLoS}
David Cox, John Little, and Donal O'Shea
\newblock An introduction to computational algebraic geometry and commutative 
algebra
\newblock {\em Undergraduate Texts in Mathematics}
\newblock Springer, New York, xvi+551, 2007.

\bibitem{CullerShalen}
Marc Culler and Peter~B. Shalen.
\newblock Varieties of group representations and splittings of {$3$}-manifolds.
\newblock {\em Ann. of Math. (2)}, 117(1):109--146, 1983.

\bibitem{AcunaMontesinos}
F.~Gonz{\'a}lez-Acu{\~n}a and Jos{\'e}~Mar{\'{\i}}a Montesinos-Amilibia.
\newblock On the character variety of group representations in {${\rm
  SL}(2,{\bf C})$} and {${\rm PSL}(2,{\bf C})$}.
\newblock {\em Math. Z.}, 214(4):627--652, 1993.

\bibitem{Harada}
Shinya Harada.
\newblock Canonical components of character varieties of arithmetic two bridge
  link complements.
\newblock arXiv:1112.3441, 2011.

\bibitem{Hat}
Allen Hatcher.
\newblock Notes on basic 3-manifold topology.
\newblock Available at
https://www.math.cornell.edu/~hatcher/3M/3Mdownloads.html


\bibitem{HLM}
Hugh~M. Hilden, Mar{\'{\i}}a~Teresa Lozano, and Jos{\'e}~Mar{\'{\i}}a
  Montesinos-Amilibia.
\newblock On the character variety of periodic knots and links.
\newblock {\em Math. Proc. Cambridge Philos. Soc.}, 129(3):477--490, 2000.

\bibitem{JS} 
William H. Jaco and Peter B. Shalen.
\newblock Seifert fibred spaces in 3-manifolds.
\newblock {\em Mem. Amer. Math. Soc.} 220, 1979.

\bibitem{J} 
Klaus Johannson.
\newblock Homotopy equivalence of 3-manifolds with boundary.
\newblock {\em Lecture Notes in Math. 761}, Springer-Verlag, Berlin, 1979.

\bibitem{Li}
Youlin Li.
\newblock 1-dominations between Montesinos knots with at least 6 rational
tangles.
\newblock {\em Topology Appl.} 157:1033--1043, 2010.

\bibitem{LubotzkyMagid}
Alexander Lubotzky and Andy~R. Magid.
\newblock Varieties of representations of finitely generated groups.
\newblock {\em Mem. Amer. Math. Soc.}, 58(336):xi+117, 1985.

\bibitem{MorganBass}
John~W. Morgan and Hyman Bass, editors.
\newblock {\em The {S}mith conjecture}, volume 112 of {\em Pure and Applied
  Mathematics}.
\newblock Academic Press, Inc., Orlando, FL, 1984.
\newblock Papers presented at the symposium held at Columbia University, New
  York, 1979.

\bibitem{PP}
Luisa Paoluzzi and Joan Porti.
\newblock Non-standard components of the character variety for a family of
  {M}ontesinos knots.
\newblock {\em Proc. Lond. Math. Soc. (3)}, 107(3):655--679, 2013.


\bibitem{P1}
Joan Porti.
\newblock Hyperbolic polygons of minimal perimeter with given angles.
\newblock {\em Geom. Dedicata}, 156:165--170, 2012.

\bibitem{P2}
Joan Porti.
\newblock Regenerating hyperbolic cone 3-manifolds from dimension 2.
\newblock {\em Ann. Inst. Fourier (Grenoble)}, 63(5):1971--2015, 2013.

\bibitem{Rong}
Yongwu Rong.
\newblock {Degree one maps between geometric 3-manifolds}.
\newblock {\em Trans. Amer. Math. Soc.}, 332(1):411--436, 1992.

\bibitem{ThurstonNotes}
William~P. Thurston.
\newblock {The Geometry and Topology of Three-Manifolds}.
\newblock (Electronic version from 2002 available at
  http://www.msri.org/publications/books/gt3m/), 1980.

\bibitem{Weil}
Andr{\'e} Weil.
\newblock Remarks on the cohomology of groups.
\newblock {\em Ann. of Math. (2)}, 80:149--157, 1964.

\end{thebibliography}

\textsc{Aix-Marseille Universit\'e, CNRS, Centrale Marseille, I2M, UMR 7373,}

\textsc{13453 Marseille, France}
	
{luisa.paoluzzi@univ-amu.fr}

\medskip

\textsc{Departament de Matem\`atiques, Universitat Aut\`onoma de Barcelona.}

\textsc{08193 Bellaterra, Spain}

{porti@mat.uab.cat}

\end{footnotesize}

\end{document}